\documentclass{article}
\usepackage{graphicx} % Required for inserting images

\usepackage[ruled,vlined]{algorithm2e}
\usepackage{graphicx}
\usepackage{amsmath,amsfonts,amssymb,latexsym}
\usepackage{enumerate}
\usepackage{setspace}
\usepackage{color}
\usepackage[cp1250]{inputenc}
\usepackage{pgf,tikz}
\usetikzlibrary{arrows,patterns}
\usetikzlibrary{matrix,positioning}
\usepackage{soul}
\usepackage{comment}

\newtheorem{thm}{Theorem}
\newtheorem{prop}[thm]{Proposition}

\newtheorem{cor}[thm]{Corollary}
\newtheorem{rem}[thm]{Remark}
\newtheorem{lemma}[thm]{Lemma}

\newtheorem{prob}{Problem}
\newtheorem{quest}{Question}
\newcommand{\QED}{$\Box$}
\def\vertex(#1){\put(#1){\circle*{1.8}}}
\def\lab(#1)#2{\put(#1){\makebox(0,0)[c]{#2}}}

\let\oldenumerate\enumerate
\renewcommand{\enumerate}{
  \oldenumerate
  \setlength{\itemsep}{0pt}
  \setlength{\parskip}{0pt}
  \setlength{\parsep}{0pt}
}

\title{On the stress transit function}
\author{Arun Anil$^1$, Manoj Changat$^1$, Tanja Dravec$^{2,3}$, Jeny Jacob$^{1,6}$,\\ Lekshmi Kamal K.Sheela$^1$, Iztok Peterin$^{4,3}$, Polona Repolusk$^2$\thanks{corresponding author, polona.repolusk@um.si}, Rishi Ranjan Singh$^5$}

\date{February 2025}

\begin{document}

\maketitle

\noindent \small 1 Department of Futures Studies, University of Kerala, Thiruvananthapuram, India-695581.\\
\small 2 University of Maribor, Faculty of Natural Sciences and Mathematics, Koro\v{s}ka 160, 2000 Maribor, Slovenia \\
\small 3 Institute of Mathematics, Physics and Mechanics, Jadranska 19, 1000 Ljubljana, Slovenia\\
\small 4 University of Maribor, Faculty of Electrical Engineering and Computer Science, Koro\v{s}ka 46, 2000 Maribor, Slovenia\\
\small 5 Department of Computer Science \& Engineering, Indian Institute of Technology, Bhilai, Chhattisgarh, India - 491002\\
\small 6 MITS, Ernakulam, APJ Abdul Kalam Technological University, Kerala, India

\begin{abstract}
The stress interval $S(u,v)$ between $u,v\in V(G)$ is the set of all vertices in a graph $G$ that lie on every shortest $u,v$-path. A set $U \subseteq V(G)$ is stress convex if $S(u,v) \subseteq U$ for any $u,v\in U$. A vertex $v \in V(G)$ is s-extreme if $V(G)-v$ is a stress convex set in $G$. The stress number $sn(G)$ of $G$ is the minimum cardinality of a set $U$ where $\bigcup_{u,v \in U}S(u,v)=V(G)$. The stress hull number $sh(G)$ of $G$ is the minimum cardinality of a set whose stress convex hull is $V(G)$. In this paper, we present many basic properties of stress intervals. We characterize s-extreme vertices of a graph $G$ and construct graphs $G$ with arbitrarily large difference between the number of s-extreme vertices, $sh(G)$ and $sn(G)$. Then we study these three invariants for some special graph families, such as graph products, split graphs, and block graphs. We show that in any split graph $G$, $sh(G)=sn(G)=|Ext_s(G)|$, where $Ext_s(G)$ is the set of s-extreme vertices of $G$. Finally, we show that for $k \in \mathbb{N}$, deciding whether $sn(G) \leq k$ is NP-complete problem, even when restricted to bipartite graphs. 


\end{abstract}

\section{Introduction and preliminaries}

A \emph{transit function} on a nonempty finite set $V$ is a function $R : V \times V \rightarrow 2^{V}$  satisfying the three transit axioms for every $x,y \in V$: (t1)  $x \in R(x,y)$; (t2) $R(x,y)=R(y,x)$; and (t3) $R(x,x)=\{x\}$. 
Prime examples of transit functions in graphs are the interval function $I(u,v)$ defined on $I: V(G)\times V(G) \longrightarrow 2^{V(G)}$ by
\begin{equation}\label{inter}
I(u,v)  =\{w\in V: w \text{ is on a shortest }u,v\text{-path} \},
\end{equation} 
induced path function $J(u,v)$ where we replace `shortest' by `induced' in (\ref{inter}) and all-paths function $A(u,v)$ where we just delete `shortest' in (\ref{inter}). The exhaustive monograph on the interval function was conducted by Mulder~\cite{mu-80}, whereas the function $J(u,v)$ was studied initially by Duchet \cite{duchet-87} and Morgana and Mulder \cite{morgana-02}, and $A(u,v)$ by  Sampathkumar \cite{sampath-84} and Changat et al. \cite{ckm-01}, all followed by many authors.  However the first systematic study of $I(u,v)$ from the transit functions perspective was presented later, also by Mulder \cite{Mulder-2008}. The cut vertex transit function \cite{chnest-2021,shchst-2025} (also on hypergraphs) and the longest path transit function \cite{chgnpe-2011} are other examples of interesting transit functions in graphs. Recently, a new transit function, named the toll function, and its axiomatic characterizations on AT-free graphs and interval graphs, was investigated in \cite{toll,KCP-2023}.

Convexity is a widely explored mathematical concept that incorporates various types of graph intervals and, consequently, transit functions for prime examples. 
%Also here is the definition is more general. 
A family $\mathcal{C}$ of subsets of a finite set $X$ is a \emph{convexity} on $X$ if $\emptyset ,X \in {\mathcal{C}}$ and $\mathcal{C}$ is closed under intersections~\cite{vel-93}. (If $X$ is not finite, then also nested unions must belong to $\mathcal{C}$.) For a transit function $R$, \emph{convex} (or $R$-\emph{convex}) sets are defined as sets $S$ for which $R(x,y) \subseteq S$ for any $x,y\in S$. In graph theory, the most investigated intervals are geodesic intervals that are built from $I(u,v)$ and monophonic intervals arising from $J(u,v)$~\cite{chmusi-05}.

Betweenness centrality is a widely used measure to identify crucial vertices in graphs representing complex systems with flow~\cite{rrs:2022}. It estimates the load or control of a vertex in a graph by computing the sum of the fraction of shortest paths passing through a given vertex~\cite{brandes}. Let $s_{uv}$ be the number of distinct shortest paths between a vertex $u$ and a vertex $v$. Let $s_{uv}(i)$ be the number of distinct shortest paths between $u$ and $v$ that are passing through vertex $i$. The \emph{betweenness centrality} of a vertex $i$ is calculated as $\sum_{|\{u,i,v\}|=3} \frac{s_{uv}(i)}{s_{uv}}$, as defined by Freeman~\cite{freeman}. A similar yet simplified measure is the \emph{stress centrality}~\cite{shim-53}, which assumes that the importance of a vertex is proportional to the number of shortest paths passing through that vertex.
Let $u,v, x\in V$. The \emph{stress of the vertex} $x$ \index{stress of a vertex} depending on a pair of vertices $u,v$, is denoted as $s_{uv}(x)$. %, is defined as the number of distinct shortest paths from $u$ to $v$ which passes through $x$. 
%The term  ``stress" is adapted from Shimbel \cite{shim-53} in 1953 on the network centrality measure of an undirected network, which he called as stress centrality. 
This score estimates the potential of a vertex, say $x$, to control the flow between a pair of vertices in a network. % and is based on the number of shortest paths between a pair of vertices passing through $x$.  %Vertices often on these shortest paths will have a higher stress centrality.

Clearly, $s_{uv}(u)= s_{uv}(v)$ which is exactly the number of all distinct shortest paths between $u$ and $v$. Further, $s_{uv}(x)=s_{vu}(x)$ for every $x\in V$. Observe also that, $0<s_{uv}(x) \leq s_{uv}(u)$ for any $x \in I(u,v)$ and $s_{uv}(x)=0$ for any $x \notin I(u,v)$. Motivated by these remarks, it is evident that the vertices with $s_{uv}(x)=s_{uv}(u)$ are critical for every shortest path between vertex $u$ and $v$. For such vertices, $\frac{s_{uv}(x)}{s_{uv}}=1$, which means that the contribution of shortest paths between vertices $u$ and $v$ in the betweenness score of vertex $x$ is 1. If communication is assumed to be occurring through the shortest paths in a communication system, vertices with $s_{uv}(x)=s_{uv}(u)$ are essential for any communication between vertex $u$ and $v$. 
We use this concept to define the stress transit function of a graph $G=(V,E)$.

The \emph{stress interval} \index{stress function} $S: V\times V\to 2^V$ is  
      $$S(u,v) = \{ x\in I(u,v): s_{uv}(x)=s_{uv}(u)\}. $$    
Hence $x \in S(u,v)$ if and only if every shortest path between $u$ and $v$ passes through $x$. From the definition of $S(u,v)$, we can see that $S(u,v) \subseteq I(u,v)$. Moreover, the stress interval represents also a transit function for a connected graph as (t1), (t2) and (t3) are clearly satisfied. However, if $G$ is not connected and $u$ and $v$ belong to different components, then $S(u,v)=\emptyset$ because $I(u,v)=\emptyset$ and (t1) does not hold.  
      
As defined above, the stress interval $S_{G}(u,v)$ between vertices $u$ and $v$ is the set of all vertices that lie on all shortest paths between $u$ and $v$ in $G$. A subset $U$ of $V(G)$ is {\it{stress convex}} (or {\it s-convex}) if $S_{G}(u,v)\subseteq U$ for all $u,v\in U$. If $U$ is s-convex, the subgraph of $G$ induced by $U$ will be called s-convex subgraph of $G$. The \emph{stress closure} of a set $U$ of vertices of a graph $G$, $S_G \left[ U \right]$, is the union of stress intervals between all pairs of vertices from $U$, that is, $S_G \left[ U \right] = \bigcup_{u,v \in U}S_G(u,v)$. A set $U$ of vertices of $G$ is a {\it stress set} in $G$ if $S_G \left[ U \right] = V (G)$. In all these definitions subscript $G$ can be omitted if the graph $G$ is clear from the context. The cardinality of a minimum stress set in a graph $G$ is called the {\it stress number} of $G$ and denoted $sn(G)$. Given a subset $U \subseteq V(G)$, the {\it s-convex hull} $\left[ U \right]$ of $U$ is the smallest s-convex set that contains $U$. We say that $U$ is a {\it stress hull set} of $G$ if $\left[ U \right]=V(G)$. The cardinality of a minimum stress hull set of $G$ is the {\it stress hull number} of $G$, denoted $sh(G)$. Vertex $v$ of a graph $G$ is \emph{stress extreme} (\emph{s-extreme} for short) if $V(G-\{v\})$ is s-convex. The set of extreme vertices of $G$ is denoted by $Ext_s(G)$. Moreover, for an s-convex set $K$ in $G$, a vertex $v$ is s-extreme in $K$ if $K-\{v\}$ is an s-convex set of $G$. Notice that all the definitions from this paragraph can be adjusted from stress function to any other transit function like interval, induced, all-paths, cut vertices, longest paths, toll walk and others. In particular, the geodetic number of a graph $G$ is defined analogue to the stress number, only that we use intervals $I(u,v)$ instead of stress intervals $S(u,v)$.

Let $G$ be a connected graph. As usual $N(v)$ denotes the \emph{open interval} $\{u\in V(G):uv\in E(G)\}$ of $v\in V(G)$. Let $X \subseteq V(G)$ be any subset of $V(G)$. Then the {\it induced subgraph} $G[X]$ is the graph whose vertex set is $X$ and whose edge set consists of all of the edges in $E(G)$ that have both endpoints in $X$. That is, for any two vertices $ u,v\in X$, $u$ and $v$ are adjacent in $G [X]$ if and only if they are adjacent in $G$. 
A vertex $v\in V(G)$ is a {\it cut vertex}, if $G-\{ v\}$ has more components than $G$. Let $x,y$ be two vertices of the same component of a graph $G$. A set $C \subseteq V(G)$ is a {\it vertex cut} separating $x$ and $y$ if $x$ and $y$ are in different components of $G-C$. A {\it clique} of a graph $G$ is an induced subgraph of $G$ isomorphic to a complete graph. We call a clique also a set $S \subseteq V(G)$ any two verices of which are adjacent in $G$. A vertex $v \in V(G)$ is \emph{simplicial} if $N(v)$ is a clique. The {\it distance} $d_G(u,v)$ between vertices $u,v\in V(G)$ is the length of a shortest $u,v$-path in $G.$ The  {\it diameter} of a graph, $diam(G)$, is defined as $ diam(G)=\max_{u,v\in V(G)}\lbrace d(u,v)\rbrace$. A one vertex graph $G\cong K_1$ is called {\it trivial}. Let $P_1,P_2$ be two $u,v$-paths in $G$. We say that these paths are {\it internally disjoint} if  $V(P_1)\cap V(P_2)=\{u,v\}$ and $u\xrightarrow{P}  x$ denotes the subpath of a path $P$ with end vertices $u$ and $x$.

The concepts of interval numbers and interval sets have been extensively studied for various transit functions in graphs.
 For example, the geodetic number, associated with the interval function $I$, has been studied in \cite{buckley-88, chartrand-2000, chartrand-2002, harary-93}. These studies include investigations in specific classes of graphs, such as median graphs \cite{bresar-2008_2}, extreme geodetic graphs \cite{cz-2002}, and the Cartesian product of graphs \cite{bresar-2008}. Dourado et al. \cite{Dourado:2010} proved that determining whether a geodetic set of size at most $k$ exists in chordal or chordal bipartite graphs is NP-complete. Similarly, the monophonic number, defined with respect to the monophonic (induced path) function, has been studied in \cite{Paluga-07, Santhakumaran-14}. Dourado et al. \cite{Dourado:2010_2} further established that decision problems related to monophonic convexity and monophonic numbers are NP-complete. Further, the toll number for the toll function and the $\Delta$-interval number for $\Delta$-convexity have been investigated in \cite{dravec-2022} and \cite{bijo-2022}, respectively.

In the next section, we describe several basic properties of stress intervals. In particularly, we fully describe stress intervals in terms of cut vertices of some special subgraphs and characterize extreme vertices of s-convex subsets. We continue with the section about stress number and stress hull number, where we show differences and similarities of both numbers. This is underlined by computing exact values of the stress number and the stress hull number of several graph families. In Section~\ref{sec4}, we show that stress interval can be determined in polynomial time and that the decision version of the stress set problem is NP-complete. We conclude the paper with several interesting questions and problems.

\section{Basic properties of stress intervals}

In this section, we present some properties of stress intervals and s-convex sets. Moreover, we characterize graphs $G$ for which $S(u,v)=I(u,v)$ holds for any $u,v \in V(G)$. We start with the following remark that is based on the distances from $u$ to different vertices of $S(u,v)$.

\begin{rem}\label{st}
There exists a unique ordering of the elements in $S(u,v)$ ($\{u,x_1,x_2,$ $ \ldots, x_n,v\}$) such that each shortest $u,v$-path visits the vertices in the same order. We denote the shortest $u,v$-path through the vertices of $S(u,v)$ as $u-x_1 - x_2 - \cdots -x_n - v$.
\end{rem}

It is known that the geodesic interval $I(u,v)$ may not always be geodesically convex for any given pair of vertices $u$ and $v$ in a graph. For example, consider the complete bipartite graph $K_{2,3}$, let $u,v,z$ be the degree two vertices, and $x,y$ be the degree three vertices, then $x,y\in I(u,v)$, $z\in I(x,y)$, but $z\notin I(u,v)$, which shows that $I(u,v)$ is not geodesically convex. Graphs in which $I(u,v)$ are always geodesically convex are known as interval monotone graphs, see \cite{mu-80}. A characterization of interval monotone graphs still remains an unsolved problem. On the other hand, we can prove that the stress intervals $S(u,v)$ are always s-convex. 

\begin{prop}
The stress interval $S(u,v)$ is s-convex for any graph $G$.
\end{prop}
\begin{proof} 
Let $x,y$ be arbitrary vertices of $S(u,v)$. By Remark~\ref{st}, we may without loss of generality assume that $x$ is before $y$ in any shortest $u,v$-path $P$, i.e.\  $P:u-x-y-v$. Let $z \in S(x,y)$. Hence $z$ lies on all shortest $x,y$-paths and since any shortest $u,v$-path contains shortest $x,y$-path as a subpath, it follows that $z$ is in any shortest $u,v$-path. Therefore $z \in S(u,v)$ and thus $S(x,y) \subseteq S(u,v)$.

%Assume that $u$ and $v$ belong to different components of $G$. Then $S(u,v)=\emptyset$ which is s-convex by the definition. Therefore, let  $u$ and $v$ belong to the same component of $G$. We claim that if $x \in S(u, v)$, then $S(u,x)\subseteq S(u,v)$. \\
%Let $y \in S(u, x)$. Since $x \in I(u, v)$, $y \in  I(u, x)$, clearly, $y \in I(u,v)$. Now, it is enough to show that $s_{uv}(y) = s_{uv}(u)$. Let $P$ be any shortest $u,v$-path. Clearly, $P$ contains $x$, since $x \in S(u,v)$. Let $P'$ be the shortest $u,x$-subpath of $P$. Since $y\in S(u,x)$,  $P'$ contains $y$ and hence  $P$ contains $y$. That is, any shortest $u,v$-path in $G$ will pass through $y$. Hence, $s_{uv}(y) = s_{uv}(u)$. Thus, for any $x \in S(u, v)$, $S(u,x)\subseteq S(u,v)$ and our claim follows.

%To prove the proposition, let $x,y \in S(u,v)$ and $z\in S(x,y)$. By Remark ~\ref{st}, either $u-x-y-v$ or $u-y-x-v$ will be a shortest path from $u$ to $v$; say the first.  Clearly $y \in I(x,v)$. If $s_{xv}(y) < s_{xv}(x)$, then there exists an $x,v$-shortest path $Q$ which does not pass through $y$. Any shortest $u,x$-path together with $Q$ will be a $u,v$-shortest path not passing through $y$. Therefore, $y\notin S(u,v)$, a contradiction. Thus, $s_{xv}(y) = s_{xv}(x)$ and $y \in S(x,v)$ follows. Since, $z\in S(x,y)$, by Claim, $z \in S(x,v)$. Since $x\in S(u,v)$ and $z \in S(x,v)$, again by our claim, we get $z \in S(u,v)$. \hfill\QED

\end{proof}

In the following theorem, we characterize stress intervals of cardinality 2. For this, we need the well-known Menger's theorem.

\begin{thm}\label{thm:Menger}\cite{Menger-1927}
Let $G$ be a connected graph and $x,y$ arbitrary non-adjacent vertices of $G$. Then the maximum number of internally disjoint paths between $x$ and $y$ equals to the minimum number of vertices in a vertex cut separating $x$ and $y$.
\end{thm}

\begin{thm}\label{p:S eq u,v}
    Let $G$ be a connected graph and let $u$ and $v$ be different vertices of $G$. Then $S(u,v) = \{u,v\}$ if and only if either $uv \in E(G)$ or there are at least two internally disjoint shortest $u,v$-paths.
\end{thm}
\begin{proof}
If $uv \in E(G)$, then $u,v$ is the only shortest $u,v$-path and hence $S(u,v)=\{u,v\}$. Suppose now that there exists two internally disjoint $u,v$-paths $P$ and $Q$. Hence no vertex of $V(G)- \{u,v\}$ belongs to every shortest $u,v$-path. Indeed, no vertex of $V(G)- \{u,v\}$ is on both $P$ and $Q$ as they are internally disjoint. Hence $S(u,v)=\{u,v\}$.

For the converse assume that $S(u,v)=\{u,v\}$ and suppose that $uv \notin E(G)$. For the purpose of a contradiction, assume that there do not exist two internally disjoint $u,v$-paths. Thus, by Menger's theorem, there exists a vertex cut separating $u$ and $v$ of cardinality 1. Let $\{x\}$ be a cut set that separates $u$ and $v$. Hence, every (shortest) $u,v$-path in $G$ contains $x$ and therefore $x \in S(u,v)$, a contradiction. \hfill\QED
\end{proof}

There are several ways to describe stress intervals. Here, we present one characterization, while another will be discussed in Section \ref{sec4}, where we will also show that the stress interval between two vertices can be computed in polynomial time. First, we present a simple lemma.

\begin{lemma}\label{l:cutVertices}
Let $G=(V,E)$ be a graph, $u,v \in V(G)$ and let $H= G[I(u,v)]$. Then for any $x \in S(u,v)-\{u,v\}$ it holds that  $H-\{x\}$ is not connected.
\end{lemma}
\begin{proof}
If $x \in S(u,v)-\{u,v\}$, then $x$ lies on every shortest $u,v$-path. If $H-\{x\}$ is connected, then there is a shortest $u,v$-path $P$ and $x\notin P$, a contradiction.\hfill\QED
\end{proof}

\begin{prop}\label{s-interval}
Let $G$ be a connected graph and $u,v\in V(G)$. Then 
$$S(u,v)=\{x \in I(u,v)~|~  x \textrm{ is a cut vertex of } G[I(u,v)]\} \cup \lbrace u,v \rbrace.$$
\end{prop}

\begin{proof}
Let $A=\{x \in I(u,v)~|~x \textrm{ is a cut vertex of } G[I(u,v)]\}$. If $A=\emptyset$, then either $uv \in E(G)$ or there are at least two internally disjoint shortest $u,v$-paths in $G[I(u,v)]$ and the result follows by Theorem \ref{p:S eq u,v}. 

Let now $A\neq\emptyset$. Clearly, $u,v\in S(u,v)$ and we need to show that $S(u,v)-\{u,v\}=A$. By Lemma \ref{l:cutVertices} it follows that $S(u,v)-\{u,v\}\subseteq A$. Moreover, if $x\in A$, then every shortest $u,v$-path passes through $x$ and $x\in S(u,v)-\{u,v\}$.\hfill\QED
\end{proof}

The next corollary emerges immediately from Proposition \ref{s-interval} and the role of cut vertices in interval function $I(u,v)$.

\begin{cor}
Let G be a graph. Then $x \in S(u,v)$ if and only if $I(u,x)\cup I(x,v)= I(u,v)$.     
\end{cor}

We continue with exploration of the potential of Theorem \ref{p:S eq u,v}, this time in the direction of all pairs of vertices. We say that a graph $G$ is an \emph{s-trivial graph} if $S(u,v)\subseteq \{u,v\}$ for all $u,v\in V(G)$. If $G$ is connected, then $G$ is s-trivial when $S(u,v)=\{u,v\}$ for all $u,v\in V(G)$, while for a disconnected graph $G$ every component of $G$ must be s-trivial itself. This can also be motivated directly. We will show that many graphs have this property. But first, observe that complete graphs $K_n$, complete graphs without an edge $K_n-e$, $n>3$, and a four-cycle $C_4$ are s-trivial.

\begin{prop}\label{s-trivial}
The following statements are equivalent for a graph $G$.
\begin{enumerate}
    \item $G$ is s-trivial.
    \item There are at least two internally disjoint shortest paths between any two non-adjacent vertices $u$ and $v$ of the same component of $G$. 
    \item  For any two vertices $u,v \in V(G)$ that are at distance 2, there exist at least two shortest paths between $u$ and $v$.
\end{enumerate}
\end{prop}

\begin{proof}
 The equivalence (a)$\Leftrightarrow$(b) follows from Theorem~\ref{p:S eq u,v} and the implication (b)$\Rightarrow$(c) is trivial.

  For the implication (c)$\Rightarrow$(a) assume that (c) holds. We have to show that $S(u,v) \subseteq \{u,v\}$ for all $u,v \in V(G)$. If $u$ and $v$ do not belong to the same component of $G$, then $S(u,v)=\emptyset$. So, we may assume that $u$ and $v$ belong to the same component. By the definition $\{u,v \} \subseteq S(u,v)$. Let $x$ be a vertex different from $u$ and $v$ with $x\in I(u,v)$. Let $P$ and $Q$ denote  a shortest $u,x$-path and a shortest $x,v$-path, respectively. Let $z_1\in V(P)$ and $z_2 \in V(Q)$ be adjacent to $x$. Since $d(z_1,z_2)=2$, by assumption there will be a vertex $w\neq x$ adjacent to both $z_1$ and $z_2$. Now, the shortest path joining $u$ with $z_1$ through $P$, the shortest $z_1, z_2$-path through $w$ together with the shortest path joining $z_2$ with $v$ will be a shortest $u,v$-path not passing through $x$. Therefore, $x\notin S(u,v)$, for any $x\neq u,v$ and  $x\in I(u,v)$. Since $S(u,v) \subseteq I(u,v)$ holds in any graph for any two vertices $u,v$, we deduce that $S(u,v)=\{u,v\}$. \hfill\QED 
\end{proof}

To get more examples of s-trivial graphs we recall the \emph{join} $G\vee H$ of graphs $G$ and $H$ which is a graph  that consists of one copy of $G$ and one copy of $H$ and all the possible edges with one end-vertex in the copy of $G$ and the other end-vertex in the copy of $H$. Notice that $K_{n,m}=N_n\vee N_m$ where $N_n$ is a graph on $n$ vertices and without edges. The following corollary is a direct consequence of (c) of Proposition \ref{s-trivial}. 

\begin{cor}\label{join}
If $G$ and $H$ are nontrivial graphs, then $G\vee H$ is s-trivial.     
\end{cor}

A join $H\vee H$ is a special case of lexicographic product $K_2\circ H$. More generally, the \emph{lexicographic product} of graphs $G$ and $H$ is a graph $G\circ H$ with $V(G\circ H)=V(G)\times V(H)$. Two vertices $(g,h)$ and $(g',h')$ are adjacent if $g=g'$ and $hh'\in E(H)$ or $gg'\in E(G)$. For more information on lexicographic product, we recommend \cite{HaIK}. Also, lexicographic products are s-trivial (besides some sporadic examples), which follows again from (c) of Proposition \ref{s-trivial}. 

\begin{cor}
If $G$ is a graph without isolated vertices and $H$ a nontrivial graph, then  $G\circ H$ is s-trivial.     
\end{cor}

Next, we characterize all the extreme vertices with respect to s-convexity. First, consider the following lemma.

\begin{lemma}\label{l:neighbors}
Let $G$ be a graph, $u,v \in V(G)$ and $a \in S(u,v)-\{u,v\}$. Then there exist $x,y \in N(a)$ such that $a \in S(x,y).$
\end{lemma}
\begin{proof}
Let $u,v \in V(G)$ and $a \in S(u,v)-\{u,v\}$. Hence $a$ lies in all shortest $u,v$-paths. Let $P$ be one shortest $u,v$-path and let $x,y$ be the neighbors of $a$ on $P$. Since $P$ is a shortest $u,v$-path, $d_G(x,y)=2$. Suppose that there exists $b \neq a$ such that $xb,yb \in E(G)$. Then $u \xrightarrow{P} xby \xrightarrow{P} v$ is a shortest $u,v$-path that does not contain $a$, a contradiction. Thus $a$ lies on all shortest $x,y$-paths and hence $a \in S(x,y)$.\hfill\QED
\end{proof}

\begin{thm}\label{thm: extreme}
Let $G$ be a graph and $K$ an s-convex set of $G$. Then $v$ is an extreme vertex of $K$ if and only if $d_{G-\{v\}}(x,y) \leq 2$ for any $x,y \in N_{G[K]}(v)$.
\end{thm}

\begin{proof}
Let $v$ be an extreme vertex of an s-convex set $K$ of $G$. Then $K-\{v\}$ is an s-convex set of $G$. Suppose that there exist $x,y\in N_{G[K]}(v)$ such that $d_{G-\{v\}}(x,y) \geq 3$. Then there is exactly one shortest $x,y$-path $xvy$ in $G$. Hence $v\in S(x,y)$, which contradicts the fact that $K-\{v\}$ is s-convex.

Now let $v \in K$ be the vertex with the property that $d_{G-\{v\}}(x,y) \leq 2$ for any $x,y \in N_{G[K]}(v)$. Let $x,y \in N_{G[K]}(v)$. Since $d_{G-\{v\}}(x,y) \leq 2$,  $v$ does not lie on all shortest $x,y$-paths and thus $v \notin S(x,y)$. By Lemma~\ref{l:neighbors}, $v \notin S(a,b)$ for any $a,b \in K-\{v\}$. Hence $K-\{v\}$ is s-convex and thus $v$ is an extreme vertex of $K$. \hfill\QED
\end{proof}

A subgraph $H$ of a graph $G$ is {\em isometric} if $d_H(u,v)=d_G(u,v)$ holds for any $u,v \in V(H)$.
s-convex subgraphs of a graph $G$ behave differently as convex sets for other well known graph convexities. For geodesic convexity, it is known that any convex subgraph is also isometric which is always induced. Moreover, if $G$ is connected, then $H$ being convex subgraph of $G$ implies that $H$ is connected. But for s-convexity this is not the case. There exist s-convex subgraphs that are not connected or such that are not isometric. For example, let $x_1x_2x_3x_4x_1$ be a four-cycle $C_4$. Then $\{x_1,x_3\}$ is an s-convex set in $G$ as there is no vertex different from $x_1,x_3$ that lies on all shortest $x_1,x_3$-paths. Clearly the subgraph of $C_4$ induced by $\{x_1,x_3\}$ is not connected and thus it is also not isometric. Moreover, the same holds for any pair of non-adjacent vertices in any s-trivial graph. But surprisingly we can prove the other implication.

\begin{prop}\label{l:isometricISconvex}
Let $G$ be a graph. If $H$ is an isometric subgraph of $G$, then $H$ is an s-convex subgraph of $G$. 
\end{prop}
\begin{proof}
Let $H$ be an isometric subgraph of a graph $G$. Suppose that there exist $u,v\in V(H)$ such that $S(u,v)\nsubseteq V(H)$. Hence there exists $x\in S(u,v)$ such that $x\in V(G)-V(H)$. Since $x \in S(u,v)$, $x$ lies on all shortest $u,v$-paths and thus there is no shortest $u,v$-path that is entirely contained in $H$. Hence $d_H(u,v)> d_G(u,v)$ which contradicts the fact that $H$ is an isometric subgraph of $G$. Therefore $S(u,v) \subseteq V(H)$ for any $u,v \in V(H)$ and thus $H$ is an s-convex subgraph of $G$. \hfill\QED
\end{proof}

A graph $G=(V,E)$ is \emph{geodetic} if there exists a unique shortest path between each two vertices of $G$ \cite{ore-1967}.
In the next result, we show that the stress function coincides with the interval function of a graph $G$ if and only if $G$ is geodetic.

\begin{prop}\label{p:S eq I}
A connected graph $G$ is geodetic if and only if $S(u,v)=I(u,v)$, for any $u,v \in V(G)$.
\end{prop}

\begin{proof}
   If $G$ is geodetic, then there exists a unique shortest path between every pair of vertices $u,v \in V(G)$, and hence $s_{uv}(x) = s_{uv}(u) = 1  $, for every $x \in I(u,v)$. Hence $S(u,v) = I(u,v)$. 

   Now, for the converse part, let $S(u,v)=I(u,v)$, for any $u,v \in V(G)$. On the way to a contradiction, suppose that $G$ is not geodetic. Then there are two non-adjacent vertices $u',v' \in V(G)$ %with $d(u',v') \geq 2$ 
   such that there are at least two distinct shortest paths between $u'$ and $v'$. From all vertices with this property, choose $u$ and $v$ with the smallest possible distance and let $P$ and $Q$ be two different shortest $u,v$-paths. %If $P$ and $Q$ are disjoint then $S(u,v)=\{ u,v\}$, implying $I(u,v) \not\subset S(u,v)$. 
   Assume there is $x \in V(G)$, where $x \neq u$ and $x$ is the closest vertex to $u$ and satisfies $x \in V(P) \cap V(Q)$. Hence $xu \in E(G)$ or  $u\xrightarrow{P}x$ and $u\xrightarrow{Q}x$ are two internally disjoint shortest $u,x$-paths. Proposition~\ref{p:S eq u,v} implies that $S(u,x) = \{ u,x\}$. Since $I(u,x)=S(u,x)$, we deduce $I(u,x)=\{u,x\}$ and hence $ux \in E(G)$. Since $P$ and $Q$ are distinct and in both paths, $u$ is the first and $x$ is the second vertex, it follows that $x\xrightarrow{P}v$ and $x\xrightarrow{Q}v$ are two distinct shortest $x,v$-paths in $G$, and $d_G(x,v)<d_G(u,v)$ which contradicts the choice of $u$ and $v$.   \hfill\QED
\end{proof}

Given a transit function $R$ on a non-empty set $V$, the \textit{underlying graph} $G_R$ of $R$ is the graph with vertex set $V$, where distinct $u$ and $v$ in $V$ are joined by an edge if and only if $|R(u,v)| = 2$. Observe that if $R$ is a transit function on $G$, then $G_R$ need not be isomorphic with $G$ \cite{Mulder-2008}. Also, in the case of stress function, $G_S$ and $G$ need not be isomorphic. 

\begin{prop}\label{p:underlyingIsomorphic}
A connected graph $G$ is isomorphic to $G_S$ if and only if $G$ is a geodetic graph.
\end{prop}

\begin{proof}
    Let $G$ be a geodetic graph. We have to show that $G \cong G_S$. Clearly $V(G)= V(G_S)$. Let $u,v \in V(G)$. If $uv \in E(G)$, then evidently $uv \in E(G_S)$. If $uv \in E(G_S)$, then $S(u,v)=\{u,v\}$. Since $G$ is geodetic, by Proposition \ref{p:S eq I}, we get $I(u,v) = S(u,v) =\{u,v\}$. That is, $uv \in E(G)$. Therefore $G \cong G_S$.

Conversely, suppose $G \cong G_S$. Assume that $G$ is not geodetic.     Then, among all pairs of non-adjacent vertices $u$ and $v$ that have at least two distinct shortest $u,v$-paths, say $P$ and $Q$, choose those with the smallest $d(u,v)$.    If $V(P)\cap V(Q)=\lbrace u,v \rbrace$, then $S(u,v)=\{ u,v\} = I(u,v)$ and $uv \in E(G)$, a contradiction. Therefore, let $x \in V(P) \cap V(Q)$ and $x \neq u,v$. We can choose $x$ so that $d(u,x)$
    %$ \geq 2$ 
    is the minimum. If $ux \in E(G)$, then $x\xrightarrow{P}v$ and $x\xrightarrow{Q}v$ are two distinct shortest $x,v$-paths in $G$, and $d_G(x,v)<d_G(u,v)$ which contradicts the choice of $u,v$.
    Therefore, $d(u,x)  \geq 2$. As we chose $x$ so that $d(u,x)$ is minimum, 
    $S(u,x) = \{ u,x\}$ and hence $ux \in E(G_S).$ But, since $d(u,x) \geq 2$, $ux \notin E(G)$, which is a contradiction with $G$ and $G_S$ being isomorphic. \hfill\QED
\end{proof}

Let $G$ be a connected graph with stress function $S$. Let $G_1 = G_S$ be the underlying graph of $S$. Then $G_1$ and $G$ need not be isomorphic. If they are not isomorphic, let $S_1$ be the stress function of $G_1$ and $G_2 =G_{S_1}$ be the underlying graph of $S_1$. Again, $G_2$ and $G_1$ need not be isomorphic. Thus, we can obtain a finite sequence of functions $S_1$, $S_2$, $\cdots,$ $S_{n-1}$, and a sequence of graphs $G_1$, $G_2$, $\cdots$ $G_n$, (where $G_1 =G_S, G_2 =G_{S_1},\dots,G_n =G_{S_{n-1}}$) from the graph $G$ where the sequence will end whenever $G_n$ is a geodetic graph.

\section{Stress number and stress hull number}

As in all other graph convexities it follows that all s-extreme vertices are contained in any stress hull set and also in any stress set. Moreover, any stress set is also a stress hull set. Thus $$|Ext_s(G)| \leq sh(G) \leq sn(G).$$

It is well-known that in terms of geodesic convexity the difference between the number of extreme vertices and the hull number can be arbitrary large~\cite{cz-2002} and also the difference between the hull number and the geodetic number can be arbitrary large~\cite{chz-2000,hj-2005}. The same holds also for monophonic convexity~\cite{titus-2016}. On the other hand, this is not known for toll convexity. Moreover, it was conjectured in~\cite{dravec-2022} that this does not hold. In the sequel, we present an example of a graph in which the difference between all three values can be arbitrarily large in terms of s-convexity.

Consider the following graph $G_{n,k}$. Assume we have a star $K_{1,n}$, $n\geq 2$, with the central vertex $v$, in which each edge is subdivided by one vertex. Let $\ell_1,\ldots , \ell_n$ be the leaves of this subdivided star. Now, for any $i \in \lbrace 1,\ldots,n \rbrace$  join $\ell_i$ with $v$ by $k\geq 3$ disjoint paths of length 3, see Figure~\ref{e:tn3} for $G_{4,5}$. By Theorem \ref{thm: extreme}, no vertex of $G_{n,k}$ is extreme. For the hull set, for every $i \in \lbrace 1,\ldots,n \rbrace$, consider paths of length three that we added between $v$ and $\ell_i$. Take one vertex on each $v,\ell_i$-path of length $3$, where two are neighbors of $\ell_i$ and the others are neighbors of $v$. These $k\cdot n$ vertices form a hull set $D$ of $G_{n,k}$. Note that by the construction of $G_{n,k}$, at least one vertex of each path between $v$ and $\ell_i$, $i\in\{1,\dots,k\}$, must be in $D$, therefore $sh(G_{n,k})=kn$. By the construction of $G_{n,k}$ we can derive that $S=D \cup \lbrace \ell_1,\ldots,\ell_n \rbrace$ is a smallest stress set of $G_{n,k}$ and therefore $sn(G_{n,k})=kn+n$, which shows that the difference between all the investigated parameters is arbitrarily large.

\begin{figure}[ht]
\centering
\begin{tikzpicture}[scale=1]

\draw (-3,0)--(3,0);
\draw (0,-3)--(0,3);
	
\filldraw [fill=black, draw=black,thick] (0,0) circle (3pt);
\filldraw [fill=black, draw=black,thick] (3,0) circle (3pt);
\filldraw [fill=black, draw=black,thick] (-3,0) circle (3pt);
\filldraw [fill=black, draw=black,thick] (0,3) circle (3pt);
\filldraw [fill=black, draw=black,thick] (0,-3) circle (3pt);
\filldraw [fill=black, draw=black,thick] (1.5,0) circle (3pt);
\filldraw [fill=black, draw=black,thick] (-1.5,0) circle (3pt);
\filldraw [fill=black, draw=black,thick] (0,1.5) circle (3pt);
\filldraw [fill=black, draw=black,thick] (0,-1.5) circle (3pt);

\foreach \x in {1,2}{
    \foreach \y in {0.4,0.8,1.2,1.6,2}{
		\draw (0,0)--(1,\y);
            \draw (3,0)--(2,\y);
            \draw (1,\y)--(2,\y);
		\filldraw [fill=black, draw=black,thick] (\x,\y) circle (2pt);
	}
 }
 \foreach \x in {-2,-1}{
    \foreach \y in {-0.4,-0.8,-1.2,-1.6,-2}{
		\draw (0,0)--(-1,\y);
            \draw (-3,0)--(-2,\y);
            \draw (-1,\y)--(-2,\y);
		\filldraw [fill=black, draw=black,thick] (\x,\y) circle (2pt);
	}
 }

 \foreach \y in {1,2}{
    \foreach \x in {-0.4,-0.8,-1.2,-1.6,-2}{
		\draw (0,0)--(\x,1);
            \draw (0,3)--(\x,2);
            \draw (\x,1)--(\x,2);
		\filldraw [fill=black, draw=black,thick] (\x,\y) circle (2pt);
	}
 }

 \foreach \y in {-1,-2}{
    \foreach \x in {0.4,0.8,1.2,1.6,2}{
		\draw (0,0)--(\x,-1);
            \draw (0,-3)--(\x,-2);
            \draw (\x,-1)--(\x,-2);
		\filldraw [fill=black, draw=black,thick] (\x,\y) circle (2pt);
	}
 }
\node[] at (0.7,-0.2) {$v$};
\node[] at (3.3,0) {$\ell_1$};
\node[] at (-3.3,0) {$\ell_3$};
\node[] at (0.4,3) {$\ell_2$};
\node[] at (-0.4,-3) {$\ell_4$};

\end{tikzpicture}
\caption{A graph $G_{4,5}$.\label{e:tn3}}
\end{figure}
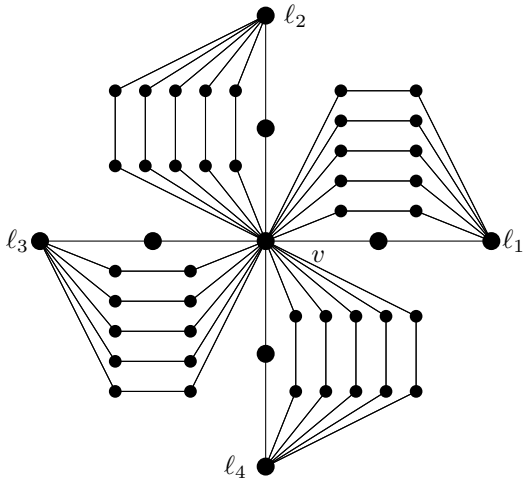

There are many graphs $G$ with  $|Ext_s(G)| = sh(G) = sn(G)$, for example, paths and complete graphs. Moreover, equality is also held in split graphs and block graphs,  which we will prove at the end of this section.

If $G$ is a nontrivial graph, then $2 \leq sh(G) \leq sn(G) \leq |V(G)|$ and all these bounds are sharp. Sharpness will be shown by computing the stress hull number and stress number of basic graph families. Since the two leaves $u$ and $v$ of $P_n$ are s-extreme and any vertex of $P_n$ lies on the only shortest $u,v$-path, it follows that $sn(P_n)=sh(P_n)=2$. It can be easily proved that paths are the only graphs with stress number and stress hull number equal to 2.

\begin{prop}\label{thm:st=2}
Let $G$ be a graph. Then $sn(G)=2$ if and only if $G$ is isomorphic to a path of order at least 2.
\end{prop}

\begin{proof}
If $G$ is a nontrivial path, then its stress number is clearly 2. Conversely, assume that $G$ is a graph with $sn(G)=2$. Then $G$ is connected, as any stress set contains at least two vertices from every component of $G$. Let $\{x,y\}$ be a stress set of $G$. Let $P$ be s shortest $x,y$-path of $G$. Suppose that there exists $v \in V(G)\setminus V(P)$. Since $\{x,y\}$ is a stress set, $v$ lies on every shortest $x,y$-path, thus $v \in P$, a contradiction. Hence $G$ is isomorphic to a path $P$, which completes the proof. \hfill\QED
\end{proof}

\begin{prop}\label{sh=2}
Let $G$ be a graph. Then $sh(G)=2$ if and only if $G$ is isomorphic to a path of order at least 2.
\end{prop}
\begin{proof}
Let $x,y$ be arbitrary vertices of $G$ and let $U=S(x,y)$ and let $a,b$ be two arbitrary vertices of $U$. Let $P$ be an arbitrary shortest $x,y$-path. Since $a,b \in S(x,y)$, it follows that $a,b \in V(P)$.  Then $S(a,b) \subseteq U$ holds for any $a,b \in U$. Indeed, if $v \in S(a,b)$, then $v$ lies on every shortest $a,b$-path and hence also on $P$. Thus $v$ is on any shortest $x,y$-path and thus $v \in U$.

What we proved is that for arbitrary two vertices $x,y \in V(G)$ it holds that $S(x,y)=[\{x,y\}]$. Hence if $\{x,y\}$ is a stress hull set of $G$, we get: $S(x,y)=[\{x,y\}]=V(G)$. Thus $\{x,y\}$ is a stress set of $G$ and thus by Proposition~\ref{thm:st=2}, $G$ is isomorphic to the path.\hfill\QED
\end{proof}

In the next result, we characterize graphs $G$ with stress number equal to the order of $G$.

\begin{thm}\label{thm:st=n}
If $G$ is a graph of order $n$, then $sn(G)=n$ if and only if $G$ is an s-trivial graph.
\end{thm}

\begin{proof}
We will use (c) of Proposition \ref{s-trivial} which is equivalent to s-trivial graphs by the same proposition. First let $sn(G)=n$. Suppose that there exists $x,y \in V(G)$ with $d(x,y)=2$ that are not contained in a $C_4$ in $G$. Then there exists exactly one shortest $x,y$-path of length 2, say $xvy$. Hence $v \in S(x,y)$ and thus $V(G)-\{v\}$ is a stress set of $G$, a contradiction. 

For the converse assume that for any two vertices $x,y \in V(G)$ that are at distance 2 in $G$, it holds that they are contained in $C_4$. If $sn(G) < n$, then there exists $w \in V(G)$ that is contained in a stress interval $S(u,v)$ for some $u,v \in V(G)-\lbrace w \rbrace$. By Lemma~\ref{l:neighbors}, there exist $x,y \in N(w)$ such that $w \in S(x,y)$. Thus $d(x,y)=2$ and since $w \in S(x,y)$ it follows that $xwy$ is the only shortest $x,y$-path. Therefore $x,y$ are not contained in $C_4$, a contradiction. \hfill\QED
\end{proof}

The next corollary follows directly from Theorem \ref{thm:st=n} and Proposition \ref{s-trivial}.

\begin{cor}\label{sn=n}
The following statements are equivalent for a graph $G$ on $n$ vertices. 
\begin{itemize}
\item $sn(G)=n$.
\item $Ext_s(G)=V(G)$.
\item $G$ is s-trivial.
\end{itemize}
\end{cor}

The \emph{Cartesian product} of graphs $G$ and $H$ is a graph $G\Box H$ with $V(G\Box H)=V(G)\times V(H)$. Two vertices $(g,h)$ and $(g',h')$ are adjacent in $G\Box H$ if $gg'\in E(G)$ and $h=h'$ or $g=g'$ and $hh'\in E(H)$, see Figure \ref{figcart} for $C_5\Box C_5$. We recommend \cite{HaIK} for more information on the Cartesian as well as other graph products. 

\begin{prop}\label{cart}
For graphs $G$ and $H$ we have $sn(G\Box H)\leq\min\{sn(G)|V(H)|,sn(H)|V(G)|\}$.
\end{prop}
\begin{proof}
We may assume that $sn(G)|V(H)|\leq sn(H)|V(G)|$ and let $D$ be a stress set of $G$. We will show that $S=D\times V(H)$ is a stress set of $G\Box H$. If $(g,h)\in V(G\Box H)-S$, then $g\notin D$. There exists $g',g''\in D$ such that $g\in S(g',g'')$ because $D$ is a stress set of $G$. Now, $(g,h)\in S((g',h),(g'',h))$ and $(g',h),(g'',h)\in S$. So, any vertex outside of $S$ belongs to a stress interval between two vertices of $S$ and thus $S$ is a stress set. Therefore the desired inequality follows.     \hfill\QED
\end{proof}

To show that this upper bound is tight we describe all s-trivial Cartesian products. For this, we need the following lemma.

\begin{lemma}\label{lemcart}
Let $G$ and $H$ be graphs and let $g,g'\in V(G)$ and $h,h'\in V(H)$. If $g\neq g'$ and $h\neq h'$, then $S((g,h),(g',h'))\subseteq\{(g,h),(g',h)\}$.
\end{lemma}
\begin{proof}
If $(g,h)$ and $(g',h')$ belongs to different components of $G\Box H$, then  
$S((g,h),(g',h'))=\emptyset$ and we are done. So, assume that $d((g,h),(g',h'))<\infty$. Let $P$ be a shortest $g,g'$-path in $G$ and let $Q$ be a shortest $h,h'$-path in $H$. Vertices $(V(P)\times\{h\})\cup (\{g'\}\times V(Q))$ and $(V(P)\times\{h'\})\cup (\{g\}\times V(Q))$ induce two shortest $(g,h),(g',h')$-paths $P'$ and $Q'$, respectively, in $G\Box H$. So, we have two shortest paths $P'$ and $Q'$ with only $(g,h)$ and $(g',h')$ in common. Hence $S((g,h),(g',h'))=\{(g,h),(g',h)\}$. \hfill\QED
\end{proof}

\begin{prop}
Let $G$ and $H$ be arbitrary graphs. Then $G \Box H$ is s-trivial if and only if $G$ and $H$ are s-trivial. 
\end{prop}
\begin{proof}
If $G\Box H$ is not s-trivial, then there exists $(g,h)\in S((g',h'),(g'',h''))-\{(g',h'),(g'',h'')\}$. By Lemma \ref{lemcart} either $h'=h''$  and we have $g\in S(g',g'')$ or $g'=g''$ and $h\in S(h',h'')$ follows. If $h'=h''$ and $g\in S(g',g'')$, then clearly $h=h'$ and $g \neq g', g \neq g''$, as $(g,h)\notin \{(g'h'),(g'',h'')\}$. Thus $G$ is not s-trivial. If $g'=g''$ and $h\in S(h',h'')$, then analogous arguments implies that $H$ is not s-trivial.

Conversely, suppose that one of $G$ and $H$, say $G$, is not s-trivial. So, there exist three different vertices $g,g',g''$ such that $g\in S(g',g'')$. For arbitrary $h\in V(H)$ we have $(g,h)\in S((g',h),(g'',h))$ and thus $G\Box H$ is not s-trivial. \hfill\QED
\end{proof}

The following consequence now easily follows from the last proposition, as well as also from Theorem~\ref{thm:st=n} because any two vertices of a hypercube $Q_n$ that are at distance 2 in $Q_n$ are contained in a $C_4$.

\begin{cor}
An $n$-dimensional hypercube $G=\Box_{i=1}^nK_2$ is s-trivial.
\end{cor}

Another example for the sharpness of the bound from Proposition \ref{cart} are grids $P_m\Box P_n$ where $P_m=x_1\ldots x_m$ and $P_n=y_1\ldots  y_n$. Let $S$ be a stress set of $P_m\Box P_n$. By Lemma \ref{lemcart} it follows that if $(x_i,y_j)\notin S$ for some $i\in\{1,\dots, m\}$ and $j\in\{1,\dots,n\}$, then we need to have at least two vertices from the set $\{(x_{\ell},y_j):\ell\in\{1,\dots,m\}\}$ or at least two vertices from $\{(x_i,y_k):k\in\{1,\dots,m\}\}$ in $S$. So, $sn(P_m\Box P_n)\geq 2\min\{m,n\}$. By Proposition \ref{cart} the equality follows. For the stress hull number of a grid, notice that $U=\{(x_1,y_1),(x_1,y_n),(x_m,y_1),(x_m,y_n)\}$ are extreme vertices of $P_m\Box P_n$ and thus $sh(P_m\Box P_n)\geq 4$. It is a straightforward observation that $[U]=V(P_m\Box P_n)$ and equality follows. 

\begin{cor}
For integers $m,n\geq 2$ we have $sn(P_m\Box P_n) = 2\cdot \min{\{m,n\}}$ and $sh(P_m\Box P_n)=4$.
\end{cor}

While it seems that many Cartesian products achieve the bound from Proposition \ref{cart}, this does not hold for odd cycles $C_{2k+1}\Box C_{2k+1}$, $k\geq 2$. For the idea observe a set of black vertices of $C_5\Box C_5$ on Figure \ref{figcart} that build a stress set of cardinality $13<15=sn(C_5)|V(C_5)|$. 

\begin{figure}[ht!]
\begin{center}
\begin{tikzpicture}[scale=0.5,style=thick,x=1cm,y=1cm]
\def\vr{4pt} % \vr = vertex radius;

\path (0,0) coordinate (x1);
\path (2,0) coordinate (x2);
\path (4,0) coordinate (x3);
\path (6,0) coordinate (x4);
\path (8,0) coordinate (x5);

\path (0,2) coordinate (y1);
\path (2,2) coordinate (y2);
\path (4,2) coordinate (y3);
\path (6,2) coordinate (y4);
\path (8,2) coordinate (y5);

\path (0,4) coordinate (z1);
\path (2,4) coordinate (z2);
\path (4,4) coordinate (z3);
\path (6,4) coordinate (z4);
\path (8,4) coordinate (z5);

\path (0,6) coordinate (u1);
\path (2,6) coordinate (u2);
\path (4,6) coordinate (u3);
\path (6,6) coordinate (u4);
\path (8,6) coordinate (u5);

\path (0,8) coordinate (v1);
\path (2,8) coordinate (v2);
\path (4,8) coordinate (v3);
\path (6,8) coordinate (v4);
\path (8,8) coordinate (v5);

\draw (x1)--(x2)--(x3)--(x4)--(x5);
\draw (y1)--(y2)--(y3)--(y4)--(y5);
\draw (z1)--(z2)--(z3)--(z4)--(z5);
\draw (u1)--(u2)--(u3)--(u4)--(u5);
\draw (v1)--(v2)--(v3)--(v4)--(v5);
\draw (x1)--(y1)--(z1)--(u1)--(v1);
\draw (x2)--(y2)--(z2)--(u2)--(v2);
\draw (x3)--(y3)--(z3)--(u3)--(v3);
\draw (x4)--(y4)--(z4)--(u4)--(v4);
\draw (x5)--(y5)--(z5)--(u5)--(v5);

\draw (x1) .. controls (3,1.5) and (5,1.5) .. (x5);
\draw (y1) .. controls (3,3.5) and (5,3.5) .. (y5);
\draw (z1) .. controls (3,5.5) and (5,5.5) .. (z5);
\draw (u1) .. controls (3,7.5) and (5,7.5) .. (u5);
\draw (v1) .. controls (3,9.5) and (5,9.5) .. (v5);
\draw (x1) .. controls (1.5,3) and (1.5,5) .. (v1);
\draw (x2) .. controls (3.5,3) and (3.5,5) .. (v2);
\draw (x3) .. controls (5.5,3) and (5.5,5) .. (v3);
\draw (x4) .. controls (7.5,3) and (7.5,5) .. (v4);
\draw (x5) .. controls (9.5,3) and (9.5,5) .. (v5);
  
\draw (x1)  [fill=black] circle (\vr); \draw (x2)  [fill=white] circle (\vr);
\draw (x3)  [fill=black] circle (\vr); \draw (x4)  [fill=white] circle (\vr);
\draw (x5)  [fill=black] circle (\vr); \draw (y1)  [fill=white] circle (\vr);
\draw (y2)  [fill=black] circle (\vr); \draw (y3)  [fill=white] circle (\vr);
\draw (y4)  [fill=black] circle (\vr); \draw (y5)  [fill=white] circle (\vr);
\draw (z1)  [fill=black] circle (\vr); \draw (z2)  [fill=white] circle (\vr);
\draw (z3)  [fill=black] circle (\vr); \draw (z4)  [fill=white] circle (\vr);
\draw (z5)  [fill=black] circle (\vr); \draw (u1)  [fill=white] circle (\vr);
\draw (u2)  [fill=black] circle (\vr); \draw (u3)  [fill=white] circle (\vr);
\draw (u4)  [fill=black] circle (\vr); \draw (u5)  [fill=white] circle (\vr);
\draw (v1)  [fill=black] circle (\vr); \draw (v2)  [fill=white] circle (\vr);
\draw (v3)  [fill=black] circle (\vr); \draw (v4)  [fill=white] circle (\vr);
\draw (v5)  [fill=black] circle (\vr); 

\end{tikzpicture}
\end{center}
\caption{A graph $C_5\Box C_5$ with a stress set (black vertices).}
\label{figcart}
\end{figure}
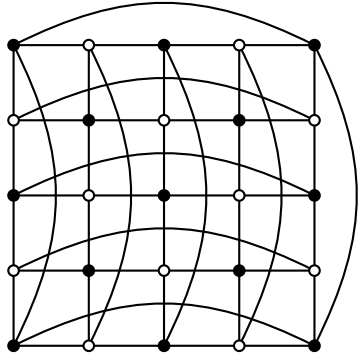

We can easily obtain stress number and stress hull number for some well-known graph families. As mentioned, a complete graph $K_n$ is s-trivial and $sn(K_n)=sh(K_n)=n$ follows from Corollary \ref{sn=n}. Similarly, the cycle $C_4$ is s-trivial and thus $sn(C_4)=sh(C_4)=4$ by the same reason. By Propositions~\ref{thm:st=2} and \ref{sh=2},
we have $3 \leq sh(C_n) \leq sn(C_n)$, $n\geq 5$. Let $C_n:v_1v_2\ldots v_nv_1$. Since $\lbrace v_1, v_{\lceil \frac{n}{2} \rceil}, v_{\lceil \frac{n}{2} \rceil+2}\rbrace$ is a stress set, we get $sh(C_n)=sn(C_n)=3$ for any $n \geq 5$. Since all leaves of a graph are extreme, $sh(K_{1,n})=sn(K_{1,n})=n$ and for $m,n\geq 2$, Corollaries~\ref{join} and \ref{sn=n} imply that $sh(K_{m,n})=sn(K_{m,n})=|V(K_{m,n})|=m+n$. For some basic graph families, their geodetic number, stress number, and stress hull number can be found in Table~\ref{tab:my_label}. %See also Figure~\ref{fig:gen}, where vertices of a stress set are colored red.

\begin{table}[ht]
    \centering
    \footnotesize
    \begin{tabular}{l|c|c|c|c}
    \hline
       Graph Type  &Order& Geodetic number & Stress number & Stress hull number \\ \hline \hline
       Path $P_n$  &$n$& 2 & 2 & 2\\ \hline
       Cycle $C_n$; $n \geq 5$  &$n$& 2 if $n$ is even, otherwise, 3 & 3 & 3  \\ \hline
    $K_n$ &$n$ & $n$ & $n$ & $n$ \\ \hline
    $K_{m,n}; m,n\geq 2$ &$m+n$ & $\min\{ m,n,4 \}$ & $m+n$ & $m+n$ \\ \hline
      $P_m \Box P_n$; $m,n \geq 2$ &$mn$ & 2  & $2\cdot \min\{m,n\}$ & 4 \\ \hline
       $Q_n$ & $2^n$& 2  & $2^n$ & $2^n$ \\ \hline
      %Split Graph (m,n,n') & $m+n$& $m+(n-n')$  & $m+(n-n')$ \\ \hline
      
    \end{tabular}
    \caption{Comparison between geodetic number, stress hull number and stress number}
    \label{tab:my_label}
\end{table}

\emph{Split graphs} are graphs whose vertices can be partitioned into a clique and an independent set. In our next result, we prove that the stress number of any split graph $G$ equals the number of extreme vertices in $G$.

\begin{thm}\label{split} 
If $G$ is a split graph, then $|Ext_s(G)|=sh(G)=sn(G)$.
\end{thm}
\begin{proof}
Let $G$ be a split graph whose vertices are partitioned into an independent set $A$ and a clique $K$. First, let $|K|=1$ and let $A_1 \subseteq A$ be the set of vertices adjacent to the vertex from $K$ and let $A_2=A-A_1$. If $|A_1|<2$, then Theorem~\ref{thm: extreme} implies that $Ext_s(G)=V(G)$ and thus $|Ext_s(G)|=sh(G)=sn(G)=|V(G)|$. If $|A_1| \geq 2$, then all vertices of $G$, except the vertex $x$ from $K$, are extreme vertices of $G$. Let $a,a' \in A_1$. Since $x \in S(a,a')$, $V(G)-K$ is a stress set and thus $|Ext_s(G)|=sh(G)=sn(G)=|V(G)|-1$.

Now, let $|K| \geq 2$. By Theorem~\ref{thm: extreme}, all vertices of $A$ are extreme. Furthermore, if $x \in K$ has no neighbors in $A$, then  it is an extreme vertex of $G$ by Theorem~\ref{thm: extreme}. Moreover, if $x \in K$ has exactly one neighbor $a$ in $A$, then by Theorem~\ref{thm: extreme}, $x$ is extreme if and only if $a$ has a neighbor in $K-\{x\}$. Finally, if $x \in K$ has at least two neighbors in $A$, then Theorem~\ref{thm: extreme}  implies that $x$ is extreme vertex of $G$ if and only if for any two distinct vertices $a_1,a_2 \in N(x) \cap A$, there exists $y\neq x$ that is adjacent to both $a_1$ and $a_2$. Let $U$ be the set of all extreme vertices of $G$, i.e.\ $U=Ext_s(G)$. Since all vertices of $A$ are extreme, we get $A \subseteq U$. We will prove that $U$ is a stress set of $G$. If $K \subseteq U$, then $U=V(G)$ and thus $|Ext_s(G)|=sh(G)=sn(G)=|V(G)|$. So, we may assume that $K \nsubseteq U$ and let $v \in V(G)-U$. Hence $v \in K$ and $v$ has at least one neighbor in $A$. 

{\bf Case 1.} $U \cap K \neq \emptyset$ and let $x \in U \cap K$.
Suppose first that $|N(v) \cap A|=1$ and let $a$ be the only neighbor of $v$ in $A$. Since $v \notin U$ ($v$ is not extreme), deg$_G(a)=1$, by Theorem~\ref{thm: extreme}, and $v \in S(a,x)$ follows. Now let $|N(v) \cap A| \geq 2$. Since $v \notin U$, there exists $a_1,a_2 \in N(v) \cap A$ with $v$ being their only common neighbor. Then $v \in S(a_1,a_2)$.  

{\bf Case 2.} $U \cap K = \emptyset$, i.e.\ $U=A$. Since no vertex of $K$ is extreme, every vertex of $K$ has at least one neighbor in $A$. Let $v \in K$. 
If $|N(v) \cap A| \geq 2$, then since $v$ is not extreme, there exists $a_1,a_2 \in N(v) \cap A$ such that $v$ is their only common neighbor. Thus $v \in S(a_1,a_2)$. Finally, let $a$ be the only neighbor of $v$ in $A$. If $|A|=1$, then $G$ is a complete graph (as any vertex of $K$ has a neighbor in $A$) which is not possible in this case. So, $|A|>1$ and $v\in S(a,a')$ where $a'$ is an arbitrary vertex from $A-\{a\}$. 

Therefore for any $v \notin U$, it follows that $v$ belongs to the stress interval between two vertices of $U$ and thus $U$ is a stress set of $G$ and hence $|Ext_s(G)|=sh(G)=sn(G)$.\hfill\QED
\end{proof}

From the proof of Theorem~\ref{split} we derive the following.

\begin{cor}
If $G$ is a split graph, then $sh(G)$ and $sn(G)$ can be computed in polynomial time.
\end{cor}

A \emph{block} in a graph $G$ is a maximal connected subgraph of $G$ that has no cut vertices. A graph $G$ is a \emph{block graph} if every block is a clique.  In the following result we show that the set of extreme vertices of a block graph is a stress set.

\begin{thm}
Let $G$ be a block graph of order $n$. If $G$ contains $k$ cut vertices, then $|Ext_s(G)|=sn(G) = sh(G) = n - k$.
\end{thm}
\begin{proof}
Let $G$ be a block graph with $|V(G)|=n$ and let $C(G)$ denote the set of all cut vertices in $G$, where $|C(G)|=k$. Further, let $U= V(G)-C(G)$, i.e.\ $U$ is a set of simplicial vertices of $G$ with $|U| = n - k$.
By Theorem~\ref{thm: extreme} any simplicial vertex of $G$ is s-extreme and any cut-vertex of $G$ is not s-extreme. Thus $U=Ext_s(G)$. Hence $U$ is contained in any stress hull set and in any stress set of $G$.

To conclude the proof it remains to show that $U$ is a stress set. Let $v \in V(G)-U=C(G)$. Let $G_1,G_2$ be connected components of $G-v$ and let $x_1 \in V(G_1)$ and $x_2\in V(G_2)$ be arbitrary simplicial vertices, not adjacent to $v$ in $G$, of $G_1$ and $G_2$, respectively (note that since $G_1,G_2$ are block graphs each of them contains at least two non-adjacent simplicial vertices). Hence $v \in S(x_1,x_2)$. Since for $i \in \{1,2\}$, $x_i$ is simplicial in $G_i$ and $x_i$ is not adjacent to $v$ in $G$, $x_i$ is also simplicial in $G$ and thus $x_i \in U$. Hence any $v \in C(G)$ is contained in the stress interval between two simplicial vertices and thus $U$ is a stress set of $G$.

%Since $G$ is a block graph, every cut vertex  $v \in C(G)=V(G)-U$ satisfies  $v \in S(x, y)$ for some $x, y \in U$. This implies that  $C(G) \subseteq S_G[U]$ and $C(G) \subseteq [U]$. As $U \cup C(G) = V(G)$, we conclude that $S_G[U] = [U] = V(G)$.
	
%Next, we show that no subset  $U' \subset V(G)$ with $|U'| < n - k$ satisfies  $S_G[U'] = [U'] = V(G)$. Suppose that such a subset  $U'$ exists. Since $|U'| < n - k$, $U'$ excludes at least one vertex of $U$, say  $u$. In block graphs, any non-cut vertex $u$ satisfies  $u \notin S(x, y)$ for any $x, y \in V(G)$ where $u \notin \{x, y\}$. Consequently, $u \notin [U']$ and $u \notin S_G[U'] $, which implies $[U'] \neq V(G)$ and $S_G[U'] \neq V(G)$. This contradicts the assumption that  $S_G[U'] = [U'] = V(G)$. Hence, no such subset  $U'$ exists, and it follows that  $sn(G) = sh(G) = n - k$.
	
%Finally, consider the s-extreme vertices of $G$. For any non-cut vertex $u \in U$,  $V(G)-\{u\}$ is an s-convex set. However, for any cut vertex $v \in C(G)$, $V(G)- \{v\}$ is not an s-convex set. Thus, $Ext_s(G) = U$, and  $|Ext_s(G)| = n - k$. This completes the proof. 
	\hfill\QED
\end{proof}

Since the number of cut vertices in a block graph can be determined in linear time, we have the following corollary.
	
\begin{cor}
If $G$ is a block graph, then $sh(G)$ and $sn(G)$ can be computed in linear time.
\end{cor}

\section{Complexity results}\label{sec4}

Given a graph $G$ with $|V(G)|=n$ and $|E(G)|=m$, the algorithm for computing $S(u,v)$ for any $u,v\in V(G)$ first involves determining the shortest paths between all pairs of vertices. This can be achieved by running a modified version of the breadth-first search (BFS) algorithm from each vertex~\cite{brandes}. The modified BFS computes the number of shortest paths from the source vertex to all other vertices in the graph. Let $\sigma_{uv}$ be the number of shortest paths between $u$ and $v$ and let $d_{uv}$ be the length of the shortest path between $u$ and $v$. The algorithm additionally stores $d_{uv}$ as well. Moreover, $S(u,v)$ for $u,v\in V(G)$ is computed as  $$S(u,v)=\{i\in V(G) \; | \; d_{uv}=d_{ui}+d_{iv} \; \text{ and }  \; \sigma_{uv}=\sigma_{ui} \cdot \sigma_{iv} \}.$$

The breadth-first traversal algorithm, as well as the modified version, takes $O(m+n)$ time. Running it from each vertex requires a total of $O(mn+n^2)$ time. Assume that a graph $G$ is connected. Then $m\geq n-1$ and therefore, w.l.o.g., the algorithm to compute all pair shortest paths, including the number of shortest paths and length of shortest paths, can be computed in $O(mn)$. Computing $S(u,v)$ for a pair of vertices based on the above-mentioned formula will require further $O(n)$ time adding to a total time of $O(mn+n)= O(mn)$ time. For all pairs of vertices, computing $S(u,v)$ will require $O(mn+n^3)= O(n^3)$ time. With this, the following theorem is proven.

\begin{thm}\label{polynom} 
One can compute all the stress intervals of a graph on $n$ vertices in $O(n^3)$ time.
\end{thm} 

The decision version of the stress set is the following problem.

\begin{center}
\fbox{\parbox{0.93\linewidth}{\noindent
{\sc STRESS SET PROBLEM}\\[.8ex]
\begin{tabular*}{.96\textwidth}{rl}
\emph{Input:} & A graph $G$ and $k\in{\mathbb{N}}$.\\
\emph{Question:} & \text{Is there a stress set of} $G$ \text{of cardinality at most} $k$?
\end{tabular*}
}}
\end{center}

Our proofs will relate STRESS SET PROBLEM to the following well-known {\sc DOMINATING SET PROBLEM}. Recall that a set of vertices $D$ of a graph $G$ is a \emph{dominating set} if every vertex in $V (G)-D$ has a neighbor in $D$.\\

\begin{center}
\fbox{\parbox{0.93\linewidth}{\noindent
{\sc DOMINATING SET PROBLEM}\\[.8ex]
\begin{tabular*}{.96\textwidth}{rl}
\emph{Input:} & A graph $G=(V,E)$ and $k\in{\mathbb{N}}$.\\
\emph{Question:} & \text{Is there a dominating set of} $G$ \text{of cardinality at most} $k$?
\end{tabular*}
}}
\end{center}

For the next result we use a similar polynomial time reduction as used in \cite{Dourado:2010} to prove that finding a geodetic set in chordal graphs is NP-complete.

\begin{thm}
The STRESS SET PROBLEM is NP-complete even when restricted to bipartite graphs.
\end{thm}

\begin{proof}
The problem is in NP, as it can be checked in polynomial time whether a given set $U$ of a bipartite graph $G$ is a stress set of $G$ by using Theorem \ref{polynom}. 

DOMINATING SET PROBLEM restricted to bipartite graphs is NP-complete \cite{bert84}. We will show a polynomial reduction from the DOMINATING SET PROBLEM for bipartite graphs to the STRESS SET PROBLEM for bipartite graphs. 

Let $(G,k)$ be an instance of DOMINATING SET PROBLEM, where $G$ is a bipartite graph and $(A,B)$ its partition. We define a graph $G'$ as follows. Let $V(G)=\lbrace v_1,\ldots,v_n \rbrace$, where $A=\lbrace v_1,\ldots,v_{\ell} \rbrace$ and  $B=\lbrace v_{\ell+1},\ldots,v_n \rbrace$. Define $V(G')=V(G) \cup \lbrace c ,d \rbrace \cup  \bigcup_{i=1}^n\lbrace a_i,b_i \rbrace$ and $$E(G')=E(G) \cup \{cd\} \cup \bigcup_{i=1}^n\lbrace v_i a_i,a_ib_i \rbrace \cup \bigcup_{i=1}^{\ell} \lbrace a_i c\rbrace \cup \bigcup_{i=\ell+1}^n \lbrace a_id\rbrace$$ (see Figure \ref{fig:bipartite} for an example). Let also $k'=k+n.$ 
Note that $G'$ is also bipartite with two partitions $$A'= A\cup \{c\}\cup \bigcup_{i=1}^{\ell} \lbrace b_i\rbrace \cup \bigcup_{i=\ell+1}^n  \lbrace a_i\rbrace  $$ and  $$B'= B\cup \{d\}\cup \bigcup_{i=1}^{\ell} \lbrace a_i\rbrace \cup \bigcup_{i=\ell+1}^n  \lbrace b_i\rbrace.$$
We will show that $G$ has a dominating set of size at most $k$ if and only if $G'$ has a stress set of size at most $k'$.

\begin{figure}[ht]
    \centering
   
\begin{tikzpicture}[scale=1]

\fill [fill opacity=0.5,fill=blue!20, draw=blue] (-1,0) ellipse (0.5cm and 2.5cm);
\fill [fill opacity=0.5,fill=red!20, draw=red] (1,0) ellipse (0.5cm and 2.5cm);

\draw[bend right=80] (-5,0) to node {}(5,0);
\draw (-1,2)--(1,0);
\draw (-1,2)--(1,-2);
\draw (-1,0)--(1,2);
\draw (-1,0)--(1,-2);
\draw (-1,-2)--(1,0);
\draw (-1,-2)--(1,-2);

\foreach \y in {0,2,-2}{
	\draw (-5,0)--(-3,\y);
	\draw (5,0)--(3,\y);
	\draw (-3,\y)--(-1,\y);
	\draw (3,\y)--(1,\y);
	\draw (-3,\y)--(-3.5,\y+1);
	\draw (3,\y)--(3.5,\y+1);
}

\filldraw [fill=black, draw=black,thick] (-5,0) circle (3pt);
\filldraw [fill=black, draw=black,thick] (5,0) circle (3pt);

 \foreach \y in {0,2,-2}{
		\filldraw [fill=blue, draw=black,thick] (-1,\y) circle (3pt);
}
 \foreach \y in {0,2,-2}{
		\filldraw [fill=red, draw=black,thick] (1,\y) circle (3pt);
}
 
\foreach \x in {-3,3}{
    \foreach \y in {0,2,-2}{
		\filldraw [fill=white, draw=black,thick] (\x,\y) circle (3pt);
	}
 } 
 
 \foreach \x in {-3.5,3.5}{
    \foreach \y in {-1,1,3}{
		\filldraw [fill=white, draw=black,thick] (\x,\y) circle (3pt);
	}
 }

\node[] at (-5.3,0) {$c$};
\node[] at (5.3,0) {$d$};
\node[] at (-2.7,2.3) {$a_1$};
\node[] at (-2.7,0.3) {$a_2$};
\node[] at (-2.7,-1.7) {$a_{\ell}$};
\node[] at (-3.1,3) {$b_1$};
\node[] at (-3.1,1) {$b_2$};
\node[] at (-3.1,-1) {$b_{\ell}$};

\node[] at (-2.7,-1) {{\Large $\vdots$}};
\node[] at (2.7,-1) {{\Large $\vdots$}};

\node[] at (2.7,2.3) {$a_{\ell+1}$};
\node[] at (2.7,0.3) {$a_{\ell+2}$};
\node[] at (2.7,-1.7) {$a_{n}$};
\node[] at (3.1,3) {$b_{\ell+1}$};
\node[] at (3.1,1) {$b_{\ell+2}$};
\node[] at (3.1,-1) {$b_n$};

\node[] at (-1,1.7) {$v_1$};
\node[] at (-1,-0.3) {$v_2$};
\node[] at (-1,-1.7) {$v_{\ell}$};

\node[] at (1.1,1.7) {$v_{\ell+1}$};
\node[] at (1.1,-0.3) {$v_{\ell+2}$};
\node[] at (1.1,-1.7) {$v_n$};

\node[] at (-1,-1) {{\Large $\vdots$}};
\node[] at (1,-1) {{\Large $\vdots$}};

\node[] at (-1,2.7) {\textcolor{blue}{$A$}};
\node[] at (1,2.7) {\textcolor{red}{$B$}};

\end{tikzpicture}

\caption{Construction of a bipartite graph $G'$ from a bipartite graph $G$.}
\label{fig:bipartite}
\end{figure}
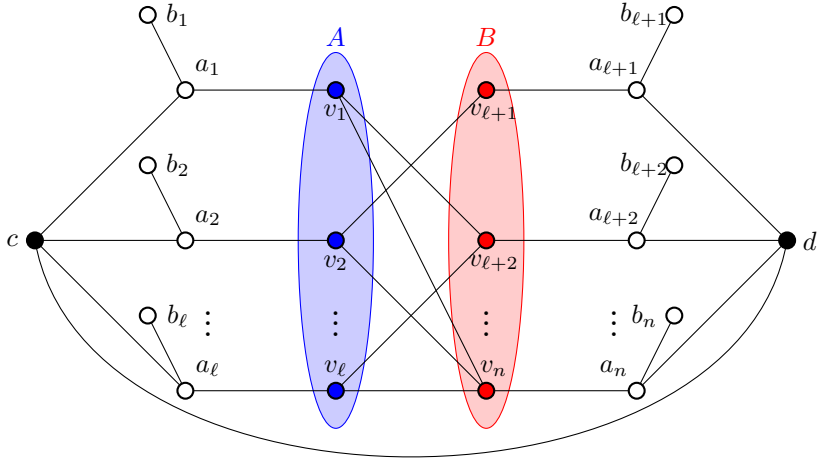

Assume first that $G$ has a dominating set of size at most $k$.  Let $D \subseteq V(G)$ be a dominating set of a graph $G$ and $|D|\leq k$. Let $X=D\cup \bigcup_{i=1}^n \lbrace b_i \rbrace$. For any $v_i,v_j \in A$ it holds that
 $S(b_i,b_j)=\{b_i,a_i,c,a_j,b_j\}$ and for any $v_i,v_j \in B$,  $S(b_i,b_j)=\{b_i,a_i,d,a_j,b_j\}$.
Therefore, $\{c,d\}\cup\{a_j;~v_j\in V(G) \} \subseteq S[\{b_j;~v_j\in V(G)\}] \subseteq S[X]$. For any $v_j \in V(G) - D$, there exists a vertex $v_i \in D$ such that $v_jv_i\in E(G)$. Hence $v_j\in S(v_i,b_j)$ and thus $\{v_j;~v_j \in V(G) -D\} \subseteq S[X]$ which implies that $X$ is a stress set of $G'$ Note that $|X|=|V(G)|+|D|\leq n+k=k'$, which completes the proof of the first implication.

To conclude the proof, let $X$ be a smallest stress set of $G'$ and $|X|\leq k'$. Let $D=X\cap V(G)$. Note that $\bigcup_{i=1}^n\lbrace b_i \rbrace \subseteq X$ as vertices $b_i$ are leaves and consequently extreme vertices of $G'$. Since $S(a_i,x) \subseteq S(b_i,x)$ holds for any $i \in \{1,\ldots , n\}$ and any $x \in V(G')$, $a_i \notin X$ by the minimality of $X$ (note that $a_i \in S(b_i,b_j)$ for any $j \neq i$). Moreover, for any $x \in V(G')$ it holds that $S(c,x) \subseteq \{c,d,a_1,\ldots ,a_n,x\}$ and all vertices from $\{c,d,a_1,\ldots ,a_n\}$ are contained in $S[\{b_1,\ldots, b_n\}]$. Thus $c,d \notin X$, again by the minimality of $X$. Hence $X=\{b_1,\ldots ,b_n\} \cup (V(G) \cap X)$ (note that $X$ contains just $b_i$'s and some vertices from $\{v_1,\ldots ,v_n\}$) and consequently $|D|=|X \cap V(G)|= |X|-n\leq k'-n=k$. Let us prove that $D$ is a dominating set of $G$. Assume conversely that $D$ is not a dominating set of $G$ and let $w\in V(G)- D$ be without a neighbor in $D$. Note that $w\notin X$, but $w\in S_{G'}\left[X\right]$. Thus there exist $x,y \in X$ such that $w\in S(x,y)$. Since $(A \cup B) \cap S(b_i,b_j)=\emptyset$ for any $i,j$, both $x$ and $y$ cannot be from $\{b_1,\ldots ,b_n\}$. Suppose now that $x=b_i$ and $y=v_j$ for some $i,j \in \{1,\ldots ,n\}$. Since $w \in S(x,y)$ it is clear that $i \in \{1,\ldots , \ell\}, j\in\{\ell+1,\ldots,n\}$ or vice versa. 
Since $w \notin X$ and $v_j \in X$ it is clear that $w \neq v_j$. Observe that $S(b_i,v_j) \cap (V(G)-\{v_j\})= \emptyset$ when $v_iv_j\notin E(G)$ and $S(b_i,v_j) \cap (V(G)-\{v_j\})=\{v_i\}$ when $v_iv_j\in E(G)$. Hence $v_iv_j \in E(G)$ and consequently $S(b_i,v_j)=\{b_i,a_i,v_i,v_j\}$ because $w \in S(b_i,v_j) \cap (V-\{v_j\})$). Thus $w=v_i$, a contradiction as $w$ has no neighbors in $X$ but $v_j \in X$ and $v_iv_j \in E(G)$. Finally, assume that $x=v_i \in X, y=v_j \in X$ and thus $w \in V(G)-\{v_i,v_j\}$. Since $w$ has no neighbors in $D=X \cap V(G)$, it is not adjacent to either of $v_i,v_j$. Hence shortest $v_i,v_j$-paths that contain $w$ are of length at least 4. But in such case (when $d(v_i,v_j) \geq 4$) no vertex of $G'$ is on every shortest $v_i,v_j$-path, thus $w \notin S(v_i,v_j)$, a final contradiction. Hence $D$ is a dominating set of $G$ of cardinality at most $k$. \hfill \QED
\end{proof}

\section{Concluding remarks}\label{sec5}

In this paper, we introduce a novel type of transit function derived from stress intervals. Following the framework of other transit functions, we define stress sets and stress hull sets, and call the cardinalities of smallest such sets as the stress number and stress hull number, respectively. We determine exact values for these parameters across various graph families and establish that the decision problem for stress sets is NP-complete, even when restricted to bipartite graphs. Additionally, we prove that the stress number for split graphs and block graphs can be computed in polynomial time. Given that the problem is NP-complete for general graphs but solvable in polynomial time for split graphs, we pose the following question regarding their superclass---the well-known chordal graphs.

\begin{quest}
What is the time complexity of the stress set problem in chordal graphs?
\end{quest}

For some special graph families, we obtained exact values also for the stress hull number, but we do not know how hard is the problem of computing the stress hull number of general graph.

\begin{quest}
Is there a polynomial-time algorithm that computes the stress hull number of a graph $G$?
\end{quest}

For some well-known types of transit function, many different graph invariants were studied. Some of them are, for example, Carath\' eodory number, Radon number and Helly number. All of these invariants would also be interesting in terms of stress transit function.

The convexity number with respect to the s-convexity could be another interesting topic. Stress convexity number of a graph $G$, $c_s(G)$, is the order of a largest proper s-convex subgraph of $G$. We can ask several questions about the s-convex sets and the stress convexity number of a given graph $G$.

\begin{prob}
Describe s-convex sets of a given graphs $G$.
\end{prob}

In many other well-known graph convexities, deciding whether the convexity number is at most $k$ (where $k \in \mathbb{N}$) is NP-complete~\cite{Dourado:2010_2, dravec-2022}. Does this also hold for s-convexity?

\begin{quest}
Is it true that for a given $k \in \mathbb{N}$, deciding whether $c_s(G) \leq k$ is NP-complete problem?
\end{quest}

Another type of problem can be derived from the underlying graphs. As briefly discussed in the last paragraph of the second section, we know that after a finite sequence of graphs $G_1, G_2,\dots, G_k$, where $(G_i)_S=G_{i+1}$ for $i\in \{1,\dots,k-1\}$, we end up with a geodetic graph $G_k$. In such a case, we say that $G_1$ converges to $G_k$. Hence, the following question and problem seem interesting.   

There exist geodetic graphs--such as trees--that cannot serve as the underlying graph of any other graph. In contrast, some block graphs can appear as the underlying graph of other graphs. For instance, a block graph formed by amalgamating two $K_4$ graphs at a single vertex also serves as the underlying graph of a graph obtained by amalgamating two $C_4$ cycles at a single vertex. This observation leads us to pose the following problem.

\begin{prob}
Characterize geodetic graphs $G$ that are the underlying graphs of a graph not isomorphic to $G$. 
\end{prob}

\begin{prob}
For a fixed geodetic graph $G$, describe all graphs that converge to $G$. 
\end{prob}

\section*{Acknowledgements}
Tanja Dravec and Iztok Peterin are partially supported by the Slovenian Research and Innovation Agency by program No. P1-0297. Rishi Ranjan Singh is partially supported by the Indian Anusandhan National Research Foundation [Grant Number: CRG/2023/007610]. Manoj Changat is partially supported by the DST, Govt. of India [Grant No. DST/INT/DAAD/P-03/2023(G)].  Arun Anil is supported by the University of Kerala, for providing the University Post Doctoral Fellowship [Ac EVII 5911/2024/UOK dated 18/07/2024].

\section*{Declaration of interests}
The authors declare that they have no conflict of interest.

\section*{Data availability}
Our manuscript has no associated data.

\end{document}